\documentclass{amsart}
\usepackage{amssymb}
\usepackage{amsmath}
\usepackage{amsfonts}

\newtheorem{theorem}{Theorem}
\theoremstyle{plain}

\newtheorem{corollary}{Corollary}

\newtheorem{lemma}{Lemma}

\newtheorem{remark}{Remark}

\numberwithin{equation}{section}
\input{tcilatex}

\begin{document}
\title[Perturbation of the mapping and solvability]{Perturbation of the
mapping and solvability theorems in the Banach space}
\author{Kamal N. Soltanov }
\address{{\small Department of Mathematics, Faculty of Sciences, Hacettepe
University, Beytepe, Ankara, TR-06532, TURKEY}}
\email{soltanov@hacettepe.edu.tr ; sultan\_kamal@hotmail.com\ }
\date{}
\subjclass[2000]{Primary 47J10, 47H14, 46E35\bigskip ; Secondary 35G15\strut
, 35Q55}
\keywords{Continuous mapping, perturbation, equation (inclusion) in Banach
space, solvability theorem}

\begin{abstract}
Here we consider a perturbation of a continuous mappings on Banach spaces,
and investigate their image under various conditions. Consequently we study
the solvability of some classes of equations and inclusions. For these we
start by the investigation of local properties of the considered mapping and
local comparing this mapping with certain smooth mappings. Moreover we study
different mixed problems.
\end{abstract}

\maketitle

\section{Introduction}

In this article we consider some perturbation of continuous mappings on
Banach spaces, and investigate their image under various conditions. Further
we study the solvability of some equations and inclusions, and also some
mixed problems. Here we obtain such results that can be used under
investigation of the various problems that generate a mapping, which is only
continuous. We would like to investigate of the obtained continuous mapping
under sufficiently less additional conditions as in [28]. It should be noted
the nonlinear mappings and equations (inclusions) were investigated in many
works earlier under different conditions (see, for example, [1-5, 7- 10,
13,14, 16, 18, 20, 21, 24 - 26, 28-33] etc.). And also the perturbations of
the mappings were investigated earlier in the different cases (see, for
example, [6, 17, 19, 21, 24, 26, 30, 32, 33] etc.). For the investigation of
the considered problem we will use several different methods.

Here for investigation of a mapping we will use the method of the locally
comparison. Namely we will compare the studied mapping with certain known
mapping, which is either studied earlier or easily studied mapping in some
sense. We will note that here the local comparing is conducted between the
considered mapping and certain continuous mapping studied in the [28]. In
other words the method used here is based to results of [28, 30, 31].
Furthermore we investigate the perturbation of the continuous mappings, the
images of which possess some known properties. Here the obtained results can
be applied to study different problems for differential equations and
inclusions. However in this article these results will be applied to the
mixed problems for only nonlinear differential equations.

Whereas we will use the main result of [28], therefore we deduce it here.

Let $X$ , $Y$ be reflexive Banach spaces and $X^{\ast }$, $Y^{\ast }$ are
their dual, respectively; and let $f:D\left( f\right) $ $\subseteq
X\rightarrow Y$ be a bounded mapping (in general, multivalued). We will note
that $f:D\left( f\right) $ $\subseteq X\rightarrow Y$ is called a bounded
mapping iff the image $f\left( G\right) $ is a bounded subset of $Y$ for any
bounded subset $G$ (in $X$) of $D\left( f\right) $. Henceforth we will
denote by $B_{r}^{X}\left( x_{0}\right) $ a closed ball and by $%
S_{r}^{X}\left( x_{0}\right) $ its boundary (i.e. of sphere) in the
respective space $X$ with center $x_{0}\in X$ and radius $r>0$.

\begin{theorem}
Let $f:D\left( f\right) $ $\subseteq X\rightarrow Y$ be a bounded mapping
and $Y$ be a reflexive Banach space with strictly convex norm together with
its dual space $Y^{\ast }$. We assume that on a closed ball $%
B_{r_{0}}^{X}\left( x_{0}\right) \subseteq D\left( f\right) $ the following
conditions hold:

$\alpha )$ $f$ is a multivalued lower semi-continuous mapping (i.e. $f\in
C_{lsc}^{0}$; if $f$ is single-valued then $f\in C^{0}$) such that $f\left(
x\right) $ is a closed convex set for any $x\in B_{r_{0}}^{X}\left(
x_{0}\right) $;

$\beta )$ there are continuous functions $\mu :R_{+}^{1}\longrightarrow
R_{+}^{1}$, $\nu :R_{+}^{1}\longrightarrow R^{1}$ nondecreasing under $\tau
\geq \tau _{0}>0$ ($\mu ,\nu \in C^{0}$) such that

$\beta _{1})$ 
\begin{equation}
\left\Vert y-y_{0}\right\Vert _{Y}\leq \mu \left( \left\Vert
x-x_{0}\right\Vert _{X}\right) ,\quad \forall y\in f\left( x\right) ,\
\forall y_{0}\in f\left( x_{0}\right) ;  \tag{0.1}
\end{equation}%
holds for any $x\in B_{r_{0}}^{X}\left( x_{0}\right) $;

$\beta _{2})$ there are $y_{0}\in f\left( x_{0}\right) $, and a mapping $%
g:D\left( g\right) \subseteq X\rightarrow Y^{\ast }$, such that $g^{-1}\in
C^{0}$, $g\left( x_{0}\right) =0$, for $U_{0}\subset D\left( g\right) $ $cl$%
\/\/$\left( g\left( U_{0}\right) \right) \equiv B_{1}^{Y^{\ast }}\left(
0\right) $, $cl$\/$\left( g\left( U_{1}\right) \right) \supseteq
S_{1}^{Y^{\ast }}\left( 0\right) $, $U_{1}\equiv g^{-1}\left( S_{1}^{Y^{\ast
}}\left( 0\right) \cap g\left( U_{0}\right) \right) $, and for any $y^{\ast
}\in g\left( U_{0}\right) \subseteq B_{1}^{Y^{\ast }}\left( 0\right) $,$\
x=g^{-1}\left( y^{\ast }\right) \in U_{0}$ 
\begin{equation*}
\left\langle y-y_{0},y^{\ast }\right\rangle \geq \nu \left( \left\Vert
x-x_{0}\right\Vert _{X}\right)
\end{equation*}%
holds for any $y\in f\left( x\right) $, moreover$\ \nu \left( r_{0}\right)
\geq \delta _{0}>0$, where $clU_{0}\equiv B_{r_{0}}^{X}\left( x_{0}\right) $%
, $U_{0}\equiv B_{r_{0}}^{X}\left( x_{0}\right) \cap D\left( g\right) $ , $%
U_{1}\equiv S_{r_{0}}^{X}\left( x_{0}\right) \cap D\left( g\right) $.

Then $f\left( B_{r_{0}}^{X}\left( x_{0}\right) \right) $ has a subset, which
is everywhere dense in some connected subset $M_{0}\subset $ $Y$ with the
nonempty interior (i.e. $M_{0}$ is a bodily subset of $Y$). Moreover this
subset has the form \ 
\begin{equation}
M_{0}\equiv \left\{ y\in Y\left\vert \ \left\langle y,y^{\ast }\right\rangle
\leq \left\langle y_{x},y^{\ast }\right\rangle ,\right. \forall y^{\ast }\in
S_{1}^{Y^{\ast }}\left( 0\right) ,\ \exists x\in S_{r_{0}}^{X}\left(
x_{0}\right) ,\forall y_{x}\in f\left( x\right) \right\} .  \tag{0.2}
\end{equation}
\end{theorem}

Now we will study the perturbation the continuous mappings of such type when
these are single-valued or multi-valued (set valued) and also we will deduce
their comparative analysis. Later we will investigate various mixed problems
with use of the main results. In particular, here we will study a
modification of the Navier-Stokes equations.

\section{On Perturbation of a Mapping}

Let $X$ , $Y$\ be a Banach spaces. We will consider the mapping determined
in the form $f\left( x\right) \equiv f_{0}\left( x\right) +f_{1}\left(
x\right) $ where $f_{0},\ f_{1}:D\left( f\right) \subseteq X\longrightarrow
Y $ are some mappings. And we will study the image $f\left( G\right) $ of a
subset $G\subseteq D\left( f\right) $. As the eventual result we will study
the following equation or inclusion 
\begin{equation}
f\left( x\right) \equiv f_{0}\left( x\right) +f_{1}\left( x\right) =y\text{
\ or \ }f\left( x\right) \equiv f_{0}\left( x\right) +f_{1}\left( x\right)
\ni y\text{, \ \ }y\in Y.  \tag{1.1}
\end{equation}%
The solvability of (1.1) will be investigated when $f_{0}$ is a regular
mapping in some sense, and $f_{1}$ is some perturbation of $f_{0}$.

As known any continuous mapping $f$ can be approximated by a regular mapping
with some exactness. Therefore we can always represent $f$ in the form $%
f\equiv f_{0}+f_{1}$ where $f_{0}$ is the main and regular part, and $f_{1}$
is a nonsmooth part of the mapping $f$. Then for study\ of the solvability
of (1.1) as usually (see, [6, 21, 24] etc.) we need find some sufficiently
conditions on the perturbation.

So let $X$ be a reflexive Banach space with strictly convex norm together
with its dual space $X^{\ast }$\footnote{%
Henceforth we assume that all a reflexive Banach space possess such norm as
it is possible (see, [4, 21, 9]).}, moreover $X^{\ast }$ is a separable
space.

\begin{theorem}
Let $f:D\left( f\right) \subseteq X\longrightarrow X^{\ast }$ be a bounded
multivalued mapping that has the form $f\equiv f_{0}+\ f_{1}$ where $f_{0}$, 
$\ f_{1}$ are mappings acting from $X$ into $X^{\ast }$, $X$ be a reflexive
Banach space. Assume that on the closed ball $B_{r_{0}}^{X}\left( 0\right) $ 
$\subseteq D\left( f\right) $ $\subset X$ the following conditions are
fulfilled:

(\textit{i}) $f_{0}\in C^{0}$ be an interior mapping (i.e. $f_{0}$ is the
single-valued continuous open mapping), $f_{0}\left( 0\right) =0$;

(\textit{ii}) $f_{1}\in C_{cm}^{0}$ (i.e. is a continuous multivalued
mapping) such that the image $f_{1}\left( x\right) $ is convex closed set
for each $x\in B_{r_{0}}^{X}\left( 0\right) $, $0\in f_{1}\left( 0\right) $
and 
\begin{equation*}
\left\{ \left\langle y,x\right\rangle \left\vert ~\exists y\right. \in
f\left( x\right) \right\} \geq k\left\langle f_{0}\left( x\right)
,x\right\rangle ,
\end{equation*}%
\begin{equation*}
\sup \left\{ \left\Vert y\right\Vert _{X^{\ast }}\left\vert ~y\right. \in
f\left( x\right) \right\} \leq \mu \left( \left\Vert x\right\Vert
_{X}\right) ,\quad \forall x\in B_{r_{0}}^{X}\left( 0\right) ,
\end{equation*}%
\begin{equation}
\left\langle f_{0}\left( x\right) ,x\right\rangle \geq \nu \left( \left\Vert
x\right\Vert _{X}\right) \left\Vert x\right\Vert _{X},\text{ a.e. }x\in
B_{r_{0}}^{X}\left( 0\right) ,\ \nu \left( r_{0}\right) \geq \delta _{0}>0 
\tag{1.2}
\end{equation}%
hold, here $k>0$ is a constant and $\mu ,\nu \in C^{0}$ ($\mu
:R_{+}^{1}\longrightarrow R_{+}^{1}$, $\nu :R_{+}^{1}\longrightarrow R^{1}$)
are continuous functions such as in Theorem 1;

(iii) almost each $\widetilde{x}\in intB_{r_{0}}^{X}\left( 0\right) $
possess a neighborhood $V_{\varepsilon }\left( \widetilde{x}\right) $, $%
\varepsilon \geq \varepsilon _{0}>0$ that for $\forall x_{1},x_{2}\in
V_{\varepsilon }\left( \widetilde{x}\right) \cap B_{r_{0}}^{X}\left(
0\right) $ there is $y_{1}\in f\left( x_{1}\right) -f\left( x_{2}\right) $
such that\ following inequation is valid 
\begin{equation}
d^{Y}\left( f\left( x_{2}\right) -f\left( x_{1}\right) ,y_{1}\right) \geq
k_{1}\left( \widetilde{x}\right) \left\Vert f_{0}\left( x_{1}\right)
-f_{0}\left( x_{2}\right) \right\Vert _{X^{\ast }},  \tag{1.3}
\end{equation}%
here $k_{1}\left( \widetilde{x},\varepsilon \right) >0$ \footnote{%
Here 
\begin{equation*}
d^{Y}\left( G_{1};\widehat{y}\right) \equiv \inf \left\{ d^{Y}\left( y;%
\widehat{y}\right) \left\vert ~y\in G_{1}\right. \right\} =\inf ~\left\{
\left\Vert y-\widehat{y}\right\Vert _{Y}\left\vert ~y\in G_{1}\right.
\right\} ,
\end{equation*}%
and $f\left( x_{1}\right) -f\left( x_{2}\right) \equiv f\left( x_{1}\right)
\backslash f\left( x_{2}\right) $.}.

Then the image $f\left( B_{r_{0}}^{X}\left( 0\right) \right) $ of the ball $%
B_{r_{0}}^{X}\left( 0\right) $ is a bodily subset of $X^{\ast }$, moreover $%
f\left( B_{r_{0}}^{X}\left( 0\right) \right) $ contains a bodily subset $%
M^{\ast }$ that has the form 
\begin{equation*}
M^{\ast }\equiv \left\{ x^{\ast }\in X^{\ast }\left\vert \ \left\langle
x^{\ast },x\right\rangle \leq \left\langle y,x\right\rangle \left\vert
~\forall y\right. \in f\left( x\right) \right. \forall x\in
S_{r_{0}}^{X}\left( 0\right) \right\} .
\end{equation*}
\end{theorem}

\begin{proof}
In beginning we will note that $f_{0}^{-1}$ is lower semi-continuous mapping
under the condition (\textit{i}) (see, [32, 2, 3]). From conditions (\textit{%
i}) and (\textit{ii}) follows, that the mapping $f$ satisfies all conditions
of Theorem 1 in the particular case when $Y\equiv X^{\ast }$ and $U\left( 
\widetilde{x}\right) \equiv B_{r_{0}}^{X}\left( 0\right) $ by virtue of
Lemma 1 of [28]. Consequently the image $f\left( B_{r_{0}}^{X}\left(
0\right) \right) $ is contains a subset $M_{0}^{\ast }$, which is dense in
some bodily subset of $X^{\ast }$ (particularly, in the set $M^{\ast }$).
Thus we must prove that the subset $M_{0}^{\ast }$ is closed, i.e. $%
\overline{M_{0}^{\ast }}\equiv clM_{0}^{\ast }\equiv M_{0}^{\ast }$ and $%
\overline{M_{0}^{\ast }}\subseteq $ $f\left( B_{r_{0}}^{X}\left( 0\right)
\right) $. For simplicity first we will prove the particular case of Theorem
2 and next we will conduct the continuation of this proof.
\end{proof}

\begin{lemma}
Let all conditions of Theorem 2 are fulfilled, furthermore let the
inequality (1.3) is fulfilled for any $x_{1},x_{2}\in B_{r_{0}}^{X}\left(
0\right) $ (i.e. $k_{1}\left( \widetilde{x},r_{0}\right) \equiv k_{1}\left(
r_{0}\right) $ for any $\widetilde{x}\in B_{r_{0}}^{X}\left( 0\right) $).
Then the statement of Theorem 2 is correct.
\end{lemma}

\begin{proof}
(of Lemma) From above-mentioned reasoning follows there exists $M_{0}^{\ast
}\subseteq f\left( B_{r_{0}}^{X}\left( 0\right) \right) $ such as above and
therefore it is enough to show that $M_{0}^{\ast }$ is closed and $%
M_{0}^{\ast }\equiv \overline{M_{0}^{\ast }}\subseteq f\left(
B_{r_{0}}^{X}\left( 0\right) \right) $.

So, let $y_{0}\in \overline{M_{0}^{\ast }}$ and $f\left( B_{r_{0}}^{X}\left(
0\right) \right) \cap \overline{M_{0}^{\ast }}\subset \overline{M_{0}^{\ast }%
}$. Then there is a sequence $\left\{ y_{m}\right\} _{m\geq 1}\subset $ $%
f\left( B_{r_{0}}^{X}\left( 0\right) \right) $ such that $%
y_{m}\Longrightarrow y_{0}$ in $X^{\ast }$ at $m\longrightarrow \infty $ ($%
X^{\ast }$ is a separable space by assumption). Hence for every $\sigma >0$
there is a number $m\left( \sigma \right) \geq 1$ such that $\left\Vert
y_{m}-y_{m+\ell }\right\Vert _{X^{\ast }}<\sigma $ holds for any $m\geq
m\left( \sigma \right) $ and $\ell >0$. We shall prove that $y_{0}\in
f\left( B_{r_{0}}^{X}\left( 0\right) \right) $.

Here we have the following possible cases: either there is an infinite part
of the sequence $\left\{ y_{m}\right\} _{m\geq 1}$, which belong the image $%
f\left( x\right) $ of $x\in B_{r_{0}}^{X}\left( 0\right) $ for an element $x$
or there is a subsequence $\left\{ y_{m_{\ell }}\right\} _{\ell \geq 1}$ of
the sequence $\left\{ y_{m}\right\} _{m\geq 1}$ such that $x_{m_{\ell }}\in
f^{-1}\left( y_{m_{\ell }}\right) \cap B_{r_{0}}^{X}\left( 0\right) $ and $%
y_{m_{\ell }}\in f\left( x_{m_{\ell }}\right) -f\left( x_{m_{\ell
+p}}\right) $, $y_{m_{\ell +p}}\in f\left( x_{m_{\ell +p}}\right) -f\left(
x_{m_{\ell }}\right) $. First case is clearly as far as $f\left( x\right) $
is closed subset for any $x\in B_{r_{0}}^{X}\left( 0\right) $. Therefore we
will consider the second case. As $\left\{ y_{m}\right\} $ is a fundamental
sequence for any $\sigma >0$ there is $m\left( \sigma \right) \geq 1$ such
that $\left\{ y_{m}\right\} _{m\geq m\left( \sigma \right) }\subset
B_{\sigma }^{X^{\ast }}\left( y_{m\left( \sigma \right) }\right) $. And as $%
f $ is continuous mapping there is $\delta \left( \sigma \right) >0$ such
that $B_{\delta \left( \sigma \right) }^{X}\left( x_{m\left( \sigma \right)
}\right) \subset B_{r_{0}}^{X}\left( 0\right) $ and $f\left( B_{\delta
\left( \sigma \right) }^{X}\left( x_{m\left( \sigma \right) }\right) \right)
\subseteq B_{\sigma }^{X^{\ast }}\left( f\left( x_{m\left( \sigma \right)
}\right) \right) $. Moreover $\left\{ y_{m}\right\} _{m\geq m\left( \sigma
\right) }\subset f\left( B_{\delta \left( \sigma \right) }^{X}\left(
x_{m\left( \sigma \right) }\right) \right) \cap B_{\sigma }^{X^{\ast
}}\left( y_{m\left( \sigma \right) }\right) $ (maybe it holds for a
subsequence of $\left\{ y_{m}\right\} _{m\geq m\left( \sigma \right) }$).
Hence we can choose the preimage of $\left\{ y_{m_{\ell }}\right\} $ from $%
B_{\delta \left( \sigma \right) }^{X}\left( \widetilde{x}\right) $. Then for
the defined above subsequence there are $x_{m_{\ell }}\in f^{-1}\left(
y_{m_{\ell }}\right) \cap B_{r_{0}}^{X}\left( 0\right) $, $\ell
\longrightarrow \infty $ and $m_{\ell }>m\left( \sigma \right) $ satisfying
of the second case such that for any $m_{\ell }$, $m_{\ell +p}$ the
inequation $d^{X^{\ast }}\left( f\left( x_{m_{\ell }}\right) -f\left(
x_{m_{\ell +p}}\right) ,y_{m_{\ell +p}}\right) $ $<\sigma $ holds by virtue
of convexity of $f\left( x_{m_{\ell }}\right) $, $f\left( x_{m_{\ell
+p}}\right) $, by fundamentality of the sequence $\left\{ y_{m}\right\} $
and continuity of the mapping $f$. Then using the condition (\textit{iii})
we obtain that 
\begin{equation*}
\left\Vert f_{0}\left( x_{m_{\ell }}\right) -f_{0}\left( x_{m_{\ell
+p}}\right) \right\Vert _{X^{\ast }}\leq kd^{X^{\ast }}\left( f\left(
x_{m_{\ell }}\right) -f\left( x_{m_{\ell +p}}\right) ,y_{m_{\ell +p}}\right)
<k\sigma
\end{equation*}%
holds for any $p\geq 1$ under the assumption of Lemma.

From here using the condition (\textit{i}) we obtain there exists a
subsequence $\left\{ x_{m_{j}}\right\} _{j\geq 1}$ in the preimage, which is
defined in the form $x_{m_{j}}\in f_{0}^{-1}\left( \widetilde{y}%
_{m_{j}}\right) \equiv f_{0}^{-1}\left( f_{0}\left( x_{m_{j}}\right) \right) 
$ and which is fundamental sequence in $X$, as $f_{0}$ is the single-valued
continuous opened mapping. Consequently there is an element $x_{0}\in X$ for
which we have $x_{m_{j}}\Longrightarrow x_{0}$ at $j\longrightarrow \infty $%
, because the sequence $\left\{ f_{0}\left( x_{m_{j}}\right) \right\} $ is a
convergent. Then $x_{0}\in B_{r_{0}}^{X}\left( 0\right) $ as the ball $%
B_{r_{0}}^{X}\left( 0\right) $ is closed.

Thus from conditions (\textit{ii}) and (\textit{i}) follows that $d^{X^{\ast
}}\left( f\left( x_{m_{j}}\right) ;\left\{ y_{0}\right\} \right)
\longrightarrow 0$ in $X^{\ast }$ and $x_{m_{j}}\Longrightarrow x_{0}\in
B_{r_{0}}^{X}\left( 0\right) $ in $X$ at $j\nearrow \infty $. Consequently,
we have $y_{0}\in f\left( x_{0}\right) \subset $ $f\left(
B_{r_{0}}^{X}\left( 0\right) \right) $ as far as $x_{0}\in
B_{r_{0}}^{X}\left( 0\right) $. Hence using the corollary 2 of [28] we
obtain the correctness of this Lemma.
\end{proof}

It is not difficult to see that the following statement is true.

\begin{corollary}
Let all conditions of Theorem 2 is fulfilled except of inequation (1.3),
instead of inequation (1.3) is fulfilled the following inequation 
\begin{equation*}
d^{Y}\left( f\left( x_{2}\right) -f\left( x_{1}\right) ,y_{1}\right) \geq
k_{1}\left\Vert f_{0}\left( x_{1}\right) -f_{0}\left( x_{2}\right)
\right\Vert _{X^{\ast }}+\psi \left( \left\Vert x_{1}-x_{2}\right\Vert
_{Y},r_{0}\right)
\end{equation*}%
for any $x_{1},x_{2}\in B_{r_{0}}^{X}\left( 0\right) $, where $Y$ is a
Banach space such that inclusion $X\subset Y$ is compact. Then the statement
of Theorem 2 is correct.
\end{corollary}

Now we can return to the proof of Theorem 2.

\begin{proof}
(Theorem 2 - continuation) For this proof we will use arguments similar to
the proof of the previous Lemma.\ Thus we assume, as in the proof of the
previous Lemma, $\left\{ y_{m}\right\} _{m\geq 1}\subset $ $f\left(
B_{r_{0}}^{X}\left( 0\right) \right) $ is a sequence such that $%
y_{m}\Longrightarrow y_{0}$ in $X^{\ast }$ at $m\longrightarrow \infty $. So
we have the following possible cases: \textit{(a)} there is an infinity part
of the sequence $\left\{ y_{m}\right\} _{m\geq 1}$, which belong the image $%
f\left( x\right) $ of $x\in B_{r_{0}}^{X}\left( 0\right) $ for an element $x$%
; \textit{(b) }there is a subsequence $\left\{ y_{m_{\ell }}\right\} _{\ell
\geq 1}$ of the sequence $\left\{ y_{m}\right\} _{m\geq 1}$ such that $%
x_{m_{\ell }}\in f^{-1}\left( y_{m_{\ell }}\right) \cap B_{r_{0}}^{X}\left(
0\right) $ and $y_{m_{\ell }}\in f\left( x_{m_{\ell }}\right) -f\left(
x_{m_{\ell +p}}\right) $, $y_{m_{\ell +p}}\in f\left( x_{m_{\ell +p}}\right)
-f\left( x_{m_{\ell }}\right) $.

After similar reasons of the proof of Lemma we obtain again that the case 
\textit{(a)} is obvious and in the case \textit{(b)} there is a subsequence
satisfying similar properties as in the proof of Lemma\textit{. }Therefore
we have - the defined above subsequence there are $x_{m_{\ell }}\in
f^{-1}\left( y_{m_{\ell }}\right) \cap B_{r_{0}}^{X}\left( 0\right) $, $\ell
\longrightarrow \infty $ and for $m_{\ell }>m\left( \sigma \right) $
satisfying of the second case. In addition the inequation $d^{X^{\ast
}}\left( f\left( x_{m_{\ell }}\right) -f\left( x_{m_{\ell +p}}\right)
,y_{m_{\ell +p}}\right) <\sigma $ holds for any $m_{\ell }$, $m_{\ell +p}$
by virtue of convexity of $f\left( x_{m_{\ell }}\right) $, $f\left(
x_{m_{\ell +p}}\right) $, by fundamentality of the sequence $\left\{
y_{m}\right\} $ and continuity of the mapping $f$. Thus we can choose an
element $\widetilde{x}\in B_{r_{0}}^{X}\left( 0\right) $ and $\delta \left(
\sigma \right) >0$ such that $B_{\delta \left( \sigma \right) }^{X}\left( 
\widetilde{x}\right) \subset B_{r_{0}}^{X}\left( 0\right) $ and $f\left(
B_{\delta \left( \sigma \right) }^{X}\left( \widetilde{x}\right) \right)
\subseteq B_{\sigma }^{X^{\ast }}\left( f\left( x_{m\left( \sigma \right)
}\right) \right) $ with use the continuity of the mapping $f$ again.
Moreover $\left\{ y_{m}\right\} _{m\geq m\left( \sigma \right) }\subset
f\left( B_{\delta \left( \sigma \right) }^{X}\left( \widetilde{x}\right)
\right) \cap B_{\sigma }^{X^{\ast }}\left( y_{m\left( \sigma \right)
}\right) $ and the condition (\textit{iii}) of Theorem 2 is fulfilled on
this neighborhood. Consequently we can complete the proof of Theorem 2 by
same way as in the proof of Lemma with using the condition (\textit{iii}) in
the determined here neighborhood.
\end{proof}

Now we can prove the following result with using Theorem 2, analogously as
Theorem 1 is used in the proof of Theorem 2.

\begin{theorem}
Let $X$, $Y$ be Banach spaces, furthermore $Y$ is a reflexive and $X^{\ast }$%
, $Y^{\ast }$ are separable spaces. We assume that $f\equiv f_{0}+\
f_{1}:D\left( f\right) \subseteq X\longrightarrow Y$ be a bounded
multivalued mapping where $f_{i}:$ $D\left( f_{i}\right) \subseteq
X\longrightarrow Y$ ($i=0,1$) are such mappings that on a closed
neighborhood $U\left( x_{0}\right) \subseteq D\left( f\right) $ of an
element $x_{0}$ the conditions (i) and (iii) of Theorem 2 are fulfilled with
respect to the neighborhood $U\left( x_{0}\right) $ and $Y$. Assume the
following condition is fulfilled.

(ii') The mapping $f_{1}\in C_{cm}^{0}$ such that the image $f_{1}\left(
x\right) $ is a convex closed set for every $x\in U\left( x_{0}\right)
\subset X$ and the inequation 
\begin{equation*}
d^{Y}\left( f\left( x\right) ;f\left( x_{0}\right) \right) \leq \mu \left(
\left\Vert x-x_{0}\right\Vert _{X}\right) ,\quad \forall x\in U\left(
x_{0}\right) \subset X;\text{ }0\in f_{1}\left( x_{0}\right) ,
\end{equation*}%
holds\footnote{%
As is known the Hausdorff distance for the subsets $G_{1}$, $G_{2}$ of $Y$ is%
\begin{equation*}
d^{Y}\left( G_{1};G_{2}\right) \equiv \max \left\{ \sup \left\{ d^{Y}\left(
G_{1};y_{2}\right) \left\vert ~y_{2}\in G_{2}\right. \right\} ,\right. 
\end{equation*}%
\begin{equation*}
\left. \sup ~\left\{ d^{Y}\left( y_{1};G_{2}\right) \left\vert ~y_{1}\in
G_{1}~\right. \right\} \right\} .
\end{equation*}%
}, furthermore there is a mapping $g:D\left( g\right) \subseteq X\rightarrow
Y^{\ast }$, $U\left( x_{0}\right) \subseteq D\left( f\right) \cap D\left(
g\right) $, $g^{-1}\in C^{0}$, $g\left( U\left( x_{0}\right) \right) \equiv
B_{1}^{Y^{\ast }}\left( 0\right) $, $g\left( x_{0}\right) =0$ $\&$ $\exists
G\subset U\left( x_{0}\right) $, $g\left( G\right) \equiv S_{1}^{Y^{\ast
}}\left( 0\right) $, $G\subseteq g^{-1}\left( S_{1}^{Y^{\ast }}\left(
0\right) \right) $ and a number $\delta _{0}>0$ such that 
\begin{equation*}
\forall x\in U\left( x_{0}\right) \ \ \exists y\in f\left( x\right) ,\ \inf
\left\{ \left\langle y,y^{\ast }\right\rangle \left\vert \ y^{\ast }\in
g\left( x\right) \right. \right\} \geq k\sup \left\{ \left\langle
f_{0}\left( x\right) ,\ y^{\ast }\right\rangle \left\vert \ y^{\ast }\in
g\left( x\right) \right. \right\} 
\end{equation*}%
\begin{equation}
\inf \left\{ \left\langle f_{0}\left( x\right) -f_{0}\left( x_{0}\right)
,g\left( x\right) \right\rangle \right\} \geq \nu \left( \left\Vert
x-x_{0}\right\Vert _{X}\right) ,\text{ a.e. }x\in U\left( x_{0}\right) , 
\tag{1.4}
\end{equation}%
hold, moreover $\nu \left( \left\Vert x-x_{0}\right\Vert \right) \geq \delta
_{0}>0$ is correct for $\forall x\in G_{1}\subseteq G,$ $cl\ g\left(
G_{1}\right) \supseteq S_{1}^{Y^{\ast }}\left( 0\right) $, where the
functions $\mu ,\nu \in C^{0}$ are such as in Theorem 1 and $k>0$ is a
constant.

Then $f\left( U\left( x_{0}\right) \right) $ is a bodily subset of $Y$,
which contains the following bodily subset 
\begin{equation}
M_{0}\equiv \left\{ y\in Y\left\vert \ \sup \left\{ \left\langle y,g\left(
x\right) \right\rangle \right\} \leq \inf \left\{ \left\langle y_{x},g\left(
x\right) \right\rangle \right\} ,\right. \ \forall x\in G,\ \exists y_{x}\in
f\left( x\right) \right\} .  \tag{1.5}
\end{equation}
\end{theorem}

The following statement immediately follows from Theorem 3, which is
analogously to Corollary 1.

\begin{corollary}
Let $Z$ be a Banach space such that $X\subset Z$ is compact. Then if all
conditions of Theorem 3 are fulfilled ecxept of the condition (iii), which
is fulfilled in the following form then Theorem 3 remain true

($iii\prime $) Almost each $\widetilde{x}\in intU\left( x_{0}\right) $
possess a neighborhood $V_{\varepsilon }\left( \widetilde{x}\right) $, $%
\varepsilon \geq \varepsilon _{0}>0$ such that for $\forall x_{1},x_{2}$ $%
\in V_{\varepsilon }\left( \widetilde{x}\right) \cap U\left( x_{0}\right) $
there is $y_{1}\in f\left( x_{1}\right) -f\left( x_{2}\right) $ such that\
the following inequality is valid 
\begin{equation}
d^{Y}\left( f\left( x_{1}\right) -f\left( x_{2}\right) ;y_{1}\right)
+\varphi \left( \left\Vert x_{1}-x_{2}\right\Vert _{Z};\varepsilon \right)
\geq k_{1}\left\Vert f_{0}\left( x_{1}\right) -f_{0}\left( x_{2}\right)
\right\Vert _{Y},  \tag{1.6}
\end{equation}%
for some contnious function $\varphi :R_{+}^{1}\longrightarrow R^{1}$ such
that $\varphi \left( 0;\varepsilon \right) =0$, $k_{1}>0$ is a constant.
\end{corollary}

\begin{corollary}
Let the conditions of Theorem 2 (or of Theorem 3) are fulfilled. In
addition, if $f_{0}$ is the bijective mapping and $f_{1}$ is single-valued
mapping then $f$ is the bijective mapping.
\end{corollary}

The proof follows from the condition \textit{(iii), }essentially.

\begin{corollary}
Let $f\equiv I+\ f_{1}:D\left( f\right) \subseteq X\longrightarrow Y\equiv X$
be a mapping, where $X$ and $f_{1}:D\left( f_{1}\right) \subseteq $ $%
X\longrightarrow X$ be such as in Theorem 2 ($Ix\equiv $ $f_{0}\left(
x\right) =x$).

We assume that all conditions of Theorem 2 are fulfilled on the closed ball $%
B_{r_{0}}^{X}\left( 0\right) \subseteq D\left( f\right) $, in this case the
condition (ii) is fulfilled with duality mapping $J:X\rightleftarrows
X^{\ast }$, i.e. 
\begin{equation*}
\sup \left\{ \left\Vert y\right\Vert _{X}\left\vert ~y\in f\left( x\right)
\right. \right\} \leq \mu \left( \left\Vert x\right\Vert _{X}\right) ,\
\forall x\in B_{r_{0}}^{X}\left( 0\right) ,
\end{equation*}%
\begin{equation*}
\forall x\in B_{r_{0}}^{X}\left( 0\right) ,\ \exists y\in f\left( x\right)
\quad \left\langle y,J\left( x\right) \right\rangle \geq k\left\langle
f_{0}\left( x\right) ,J\left( x\right) \right\rangle \equiv k\left\langle
x,J\left( x\right) \right\rangle
\end{equation*}%
where $k>0\ $be a constant. Then $f\left( B_{r_{0}}^{X}\left( 0\right)
\right) $ is a bodily subset of $X$.

Furthermore if the conditions of this Corollary are fulfilled on $X$, then $%
f $ is a surjection, and if in addition $f_{1}$ is single-valued mapping
then the equation $f\left( x\right) =y$ is uniquely solvable for any $y\in X$%
.
\end{corollary}

For the proof enough to note that conditions of Theorem 2 are fulfilled
under the conditions of Corollary. Hence we deduce Corollary by applying
Theorem 2.

The following fixed-point theorem follows from Corollary 4.

\begin{corollary}
Let $X$ be a reflexive separable Banach space and $f_{1}:D\left(
f_{1}\right) \subseteq X\longrightarrow X$ be a bounded multivalued
continuous mapping, in general. Let on a closed ball $B_{r_{0}}^{X}\left(
x_{0}\right) \subseteq D\left( f_{1}\right) $ of an element $x_{0}\in
D\left( f_{1}\right) $ the mapping $f\equiv I-f_{1}$ satisfies all
conditions of Corollary 4 with a duality mapping $J:X\rightleftarrows
X^{\ast }$, i.e. in this case the inequality corresponding to (1.4) has the
form%
\begin{equation*}
\forall x\in B_{r_{0}}^{X}\left( x_{0}\right) \ \exists y\in f\left(
x\right) ,\quad \left\langle y,J\left( x-x_{0}\right) \right\rangle \geq \nu
\left( \left\Vert x-x_{0}\right\Vert _{X}\right) \left\Vert
x-x_{0}\right\Vert _{X}.
\end{equation*}

Then the mapping $f_{1}$ possess a fixed point on the ball $%
B_{r_{0}}^{X}\left( x_{0}\right) $.
\end{corollary}

It should be noted that in this case we can select the duality mapping such
that $\left\Vert J\left( x\right) \right\Vert _{X^{\ast }}\equiv \left\Vert
x\right\Vert _{X}$ and $\left\langle x,J\left( x\right) \right\rangle \equiv
\left\Vert x\right\Vert _{X}^{2}$ for any $x\in B_{r_{0}}^{X}\left( 0\right) 
$ (see, for example the corresponding corollary of [28].)

It easily see that condition (\textit{iii}) under other conditions allow to
prove completeness the image of the considered mapping that follows from the
proof of Theorem 2.

Now, we will conduct some other sufficient conditions under which the image
of a bounded continuous mapping is a closed subset.

\begin{lemma}
Let $X,Y$ be a reflexive Banach spaces, and let $f:D\left( f\right)
\subseteq X\longrightarrow Y$ be a bounded continuous mapping such that $%
f\left( G\right) $ is weakly closed if $G\subseteq D\left( f\right) $ is a
closed convex subset of $X$. Then $f\left( G\right) $ is a closed subset of $%
Y$ if $G\subseteq D\left( f\right) $ is a bounded closed convex subset of $X$%
.
\end{lemma}

The proof follows from the proof of Theorem 2. Furthermore easily to see the
correctness of the following statement.

\begin{lemma}
Let $X,Y$ are Banach spaces and a bounded continuous mapping $f:D\left(
f\right) \subseteq X\longrightarrow Y$ has a weakly closed graph, in
addition for each bounded subset $M\subseteq R\left( f\right) $ of $Y$ the
preimage $f^{-1}\left( M\right) $ is a bounded subset of $X$ . Then if $X$
is a reflexive space then $f$ is a weakly closed mapping.
\end{lemma}

\section{Comparative Analysis of Mappings on Banach spaces}

Now we will investigate properties of a mappings with use of the local
comparison between studied and known mappings. We will account here the
smooth mappings, in some sense, by known mappings. Here we will consider
essentially a single-valued mappings.

\begin{theorem}
Let $X$ , $Y$ be a reflexive Banach spaces and $X^{\ast }$, $Y^{\ast }$
their dual spaces, respectively, and let $f:D\left( f\right) \subseteq
X\longrightarrow Y$ \ be a bounded continuous single-valued mapping.

Assume that for an $x_{0}\in D\left( f\right) $ there is a mapping $%
f_{x_{0}}\equiv f_{0}:D\left( f\right) \subseteq X\longrightarrow Y$ of the
class $C^{0}$ such that on a neighbourhood $U_{\varepsilon _{0}}\left(
x_{0}\right) \subset D\left( f\right) $ satisfies the following conditions

(iv) $f_{0}\left( U_{\varepsilon _{0}}\left( x_{0}\right) \right) $ is a
bodily subset of $Y$, $f_{0}\left( x_{0}\right) \in f\left( x_{0}\right) $
and $f_{0}\left( V\left( x_{0}\right) \right) \equiv B_{r}^{Y}\left(
f_{0}\left( x_{0}\right) \right) $ for some neighbourhood $V\left(
x_{0}\right) \subseteq U_{\varepsilon _{0}}\left( x_{0}\right) $ in addition
the inverse mapping $f_{0}^{-1}$ is lower or upper semi-continuous as $%
f_{0}^{-1}:B_{r}^{Y}\left( f_{0}\left( x_{0}\right) \right) \longrightarrow
X $;

(v) there is a nondecreasing continuous function $k_{x_{0}}:\Re
_{+}\longrightarrow \Re _{+}$, $k_{x_{0}}\left( \tau \right) \geq 0$, $%
k_{x_{0}}\left( 0\right) =0$ such that 
\begin{equation}
\left\Vert f\left( x_{1}\right) -f\left( x_{2}\right) -f_{0}\left(
x_{1}\right) +f_{0}\left( x_{2}\right) \right\Vert _{Y}\leq k_{x_{0}}\left(
\left\Vert f_{0}\left( x_{1}\right) -f_{0}\left( x_{2}\right) \right\Vert
_{Y}\right)  \tag{3.1}
\end{equation}%
holds for any pair $x_{1},x_{2}\in V\left( x_{0}\right) $ and $%
k_{x_{0}}^{m_{0}}\left( \tau \right) \leq \sigma \tau $ for some $m_{0}\geq
1 $, $0<\sigma <1$.

Then there is a number $\delta =\delta \left( k_{x_{0}}\left( \varepsilon
_{0}\right) \right) $ such that a ball $B_{r_{1}}^{Y}\left( f\left(
x_{0}\right) \right) $ contains in the image $f\left( V\left( x_{0}\right)
\right) $ for a $r_{1}\in \left( 0,r\left( 1-\delta \right) \right) $.
\end{theorem}

\begin{remark}
If in Theorem 4 $f\in C^{1}$ on the neighborhood $U_{\varepsilon _{0}}\left(
x_{0}\right) \subset D\left( f\right) $, and the linear continuous operator $%
Df\left( x_{0}\right) $ such that $Im\left( Df\left( x_{0}\right) \right)
\equiv Y$ and $\left\Vert Df\left( x_{0}\right) \right\Vert
_{X\longrightarrow Y}>0$ then we can choose $f_{0}\equiv Df\left(
x_{0}\right) $ on $U_{\varepsilon _{0}}\left( x_{0}\right) $.
\end{remark}

Really, let $f_{0}\equiv Df\left( x_{0}\right) :X\longrightarrow Y$ is
linear continuous operator then there is an element $\widetilde{x}\in
V\left( x_{0}\right) $ such that 
\begin{equation*}
\left\Vert f\left( x_{1}\right) -f\left( x_{2}\right) -Df\left( x_{0}\right)
\left( x_{1}-x_{2}\right) \right\Vert _{Y}\leq
\end{equation*}%
\begin{equation*}
\left\Vert f\left( x_{1}\right) -f\left( x_{2}\right) -Df\left( \widetilde{x}%
\right) \left( x_{1}-x_{2}\right) \right\Vert _{Y}+
\end{equation*}%
\begin{equation*}
\left\Vert \left[ Df\left( \widetilde{x}\right) -Df\left( x_{0}\right) %
\right] \left( x_{1}-x_{2}\right) \right\Vert _{Y}\leq
\end{equation*}%
\begin{equation*}
k_{1}\left( \varepsilon _{0}\right) \left\Vert x_{1}-x_{2}\right\Vert
_{X}+k_{2}\left( \varepsilon _{0}\right) \left\Vert x_{1}-x_{2}\right\Vert
_{X}\leq
\end{equation*}%
\begin{equation*}
k_{1}\left( \varepsilon _{0}\right) \left\Vert x_{1}-x_{2}\right\Vert
_{X}+k_{2}\left( \varepsilon _{0}\right) \left\Vert Df\left( x_{0}\right)
\left( x_{1}-x_{2}\right) \right\Vert _{Y}
\end{equation*}%
for any $x_{1},x_{2}\in V\left( x_{0}\right) \subseteq U_{\varepsilon
_{0}}\left( x_{0}\right) $ by virtue of continuity of $Df\left( x\right) $
at $x$\ (as $f\in C^{1}$ on $V\left( x_{0}\right) $) and invertibility of
the operator $Df\left( x_{0}\right) $.

From here follows that there is $k_{0}\left( \varepsilon _{0}\right) >0$
such that 
\begin{equation*}
\left\Vert f\left( x_{1}\right) -f\left( x_{2}\right) -Df\left( x_{0}\right)
\left( x_{1}-x_{2}\right) \right\Vert _{Y}\leq k_{0}\left( \varepsilon
_{0}\right) \left\Vert Df\left( x_{0}\right) \left( x_{1}-x_{2}\right)
\right\Vert _{Y}
\end{equation*}%
holds. It is clear that with choosing $\varepsilon _{0}>0$ we can make the
coefficient $k_{0}\left( \varepsilon _{0}\right) >0$ such that the
inequation $k_{0}\left( \varepsilon _{0}\right) <1$ will satisfied. In
particular the result of [16] follows from here.

\begin{proof}
(of Theorem 4) For simplicity we assume $f\left( x_{0}\right) =f_{0}\left(
x_{0}\right) =0$. Let $y\in B_{r_{1}}^{Y}\left( 0\right) $ $\equiv
f_{0}\left( V\left( x_{0}\right) \right) $ is an element. We will show that $%
y\in f\left( V\left( x_{0}\right) \right) $, i.e. the equation $f\left(
x\right) =y$ is solvable in $V\left( x_{0}\right) $ for any $y\in
B_{r_{1}}^{Y}\left( 0\right) $ for some $r_{1}\in \left( 0,r\left( 1-\delta
\right) \right) $.

For this we will use the method of successive approximations, which is
constructed by recurrence formula (cf. [16]) 
\begin{equation}
\left\{ 
\begin{array}{c}
y_{m-1}=f_{0}\left( x_{m}\right) -f_{0}\left( x_{m-1}\right) , \\ 
y_{m}=f_{0}\left( x_{m}\right) -f_{0}\left( x_{m-1}\right) -f\left(
x_{m}\right) +f\left( x_{m-1}\right)%
\end{array}%
\right. ,\quad m=1,2,...  \tag{3.2}
\end{equation}%
where $x_{0}$ is the above element and $y_{0}$ is an element of $f\left(
V\left( x_{0}\right) \right) $. Solvability of the system (3.2) is obviously
for each chosen pair ($x_{0}$, $y_{0}$) when $x_{0}\in V\left( x_{0}\right) $
and $y_{0}\in f\left( V\left( x_{0}\right) \right) $.

Now we need to show that obtained here sequences $\left\{ x_{m}\right\} $
and $\left\{ y_{m}\right\} $ are convergent. Really from second equation of
(3.2) we have%
\begin{equation*}
\left\Vert y_{m}\right\Vert _{Y}\leq k\left( \left\Vert f_{0}\left(
x_{m}\right) -f_{0}\left( x_{m-1}\right) \right\Vert _{Y}\right)
\end{equation*}%
by condition (\textit{v}). Hence follows 
\begin{equation*}
\left\Vert f_{0}\left( x_{m}\right) -f_{0}\left( x_{m-1}\right) \right\Vert
_{Y}=\left\Vert y_{m-1}\right\Vert _{Y}\leq k\left( \left\Vert f_{0}\left(
x_{m-1}\right) -f_{0}\left( x_{m-2}\right) \right\Vert _{Y}\right) .
\end{equation*}%
Then we obtain the following expression with sequential application of this
inequality 
\begin{equation*}
\left\Vert f_{0}\left( x_{m}\right) -f_{0}\left( x_{m-1}\right) \right\Vert
_{Y}\leq k\left( ...\left( k\left( \left\Vert f_{0}\left( x_{1}\right)
-f_{0}\left( x_{0}\right) \right\Vert _{Y}\right) \right) ...\right) =
\end{equation*}%
\begin{equation*}
k^{m-1}\left( \left\Vert f_{0}\left( x_{1}\right) -f_{0}\left( x_{0}\right)
\right\Vert _{Y}\right) \leq \sigma ^{s}\left\Vert f_{0}\left( x_{1}\right)
-f_{0}\left( x_{0}\right) \right\Vert _{Y}
\end{equation*}%
for any $m:sm_{0}+1$, $s=1,2,...$. From here follows 
\begin{equation}
\left\Vert f_{0}\left( x_{m}\right) -f_{0}\left( x_{m-1}\right) \right\Vert
_{Y}\leq \sigma ^{s}\left\Vert y_{0}\right\Vert _{Y}\equiv \sigma
^{s}\left\Vert y\right\Vert _{Y},  \tag{3.3}
\end{equation}%
and also 
\begin{equation*}
\left\Vert f_{0}\left( x_{m}\right) \right\Vert _{Y}\leq \sigma
^{s}\left\Vert y\right\Vert _{Y}+\left\Vert f_{0}\left( x_{m-1}\right)
\right\Vert _{Y}\leq ...\leq
\end{equation*}%
\begin{equation*}
\underset{k=0}{\overset{s}{\sum }}\sigma ^{k}\left\Vert y\right\Vert _{Y}~=%
\frac{1-\sigma ^{s+1}}{1-\sigma }\left\Vert y\right\Vert _{Y}.
\end{equation*}%
Hence we claim that the sequence $\left\{ f_{0}\left( x_{m}\right) \right\}
_{m=1}^{\infty }$ belongs to a bounded subset of $Y$ and at that $\left\Vert
f_{0}\left( x_{m}\right) \right\Vert _{Y}\leq r_{1}$, $m=1,2,...$. Moreover
we obtain, that $\left\{ y_{m}\right\} $ and $\left\{ f_{0}\left(
x_{m}\right) \right\} _{m=1}^{\infty }$ are convergent sequences and there
is $\widetilde{y}\in B_{r}^{Y}\left( 0\right) $ such that $%
y_{m}\Longrightarrow 0$ and $f_{0}\left( x_{m}\right) \Longrightarrow 
\widetilde{y}$ by virtue of (3.3).

Now if we will take account condition (\textit{iv}) then there exists a
subsequence $\left\{ x_{m_{k}}\right\} _{1}^{\infty }$ of $\left\{
x_{m}\right\} _{m=1}^{\infty }$ and an element $\widetilde{x_{0}}\in V\left(
x_{0}\right) $ such that $x_{m_{k}}\Longrightarrow \widetilde{x_{0}}$ in $X$%
, and consequently \ $f_{0}\left( x_{m_{k}}\right) \Longrightarrow \
f_{0}\left( \widetilde{x_{0}}\right) \equiv \widetilde{y}$ because of
continuity of the mapping $\ f_{0}$.

Further from (3.2) follows $y_{m-1}-y_{m}=f\left( x_{m}\right) -f\left(
x_{m-1}\right) $, therefore we have 
\begin{equation*}
y_{0}-y_{m}\equiv y-y_{m}=f\left( x_{m}\right) -f\left( x_{0}\right) .
\end{equation*}%
Then we obtain $y-y_{m}=f\left( x_{m}\right) $ as $f\left( x_{0}\right) =0$.
If we pass to the limit at $m\longrightarrow \infty $ then we obtain, that $%
f\left( x_{m}\right) \Longrightarrow y$ because $y_{m}\Longrightarrow 0$,
i.e. $f\left( \widetilde{x_{0}}\right) \equiv y$.

Thus we showed that for any $y\in B_{r_{1}}^{Y}\left( 0\right) ,$ $\
r_{1}<\left( 1-\sigma \right) r$ the equation $f\left( x\right) \equiv y$ is
solvable in $V\left( x_{0}\right) \subseteq U_{\varepsilon }\left(
x_{0}\right) $, i.e. $B_{r_{1}}^{Y}\left( 0\right) \subseteq f\left( V\left(
x_{0}\right) \right) $.
\end{proof}

For study a multivalued mappings we can use this result, therefore we will
consider the following conditions.

(1) Let $f:D\left( f\right) \subseteq X\longrightarrow X^{\ast }$ be a
bounded multivalued lower semi-continuous mapping (i.e. $f\in C_{lsc}\left(
D\left( f\right) ;X^{\ast }\right) $), the image $f\left( x\right) $ is a
convex closed subset of $X^{\ast }$ of each $x\in D\left( f\right) $, and $%
G\subseteq D\left( f\right) $ is a convex body of $X$;

(2) Let for an $x_{0}\in G$ there is a single-valued mapping $%
f_{x_{0}}\equiv f_{0}:D\left( f_{0}\right) \subseteq X\longrightarrow
X^{\ast }$ of the class $C^{0}$, $f_{0}\left( x_{0}\right) \in f\left(
x_{0}\right) $\ such that $f_{0}\left( U_{\varepsilon _{0}}\left(
x_{0}\right) \right) $ is a bodily subset of $X^{\ast }$ for a neighbourhood 
$U_{\varepsilon _{0}}\left( x_{0}\right) \subseteq G\cap D\left(
f_{0}\right) $, in addition $f_{0}\left( V\left( x_{0}\right) \right) \equiv
B_{r}^{Y}\left( f_{0}\left( x_{0}\right) \right) $ for some subset $V\left(
x_{0}\right) \subseteq U_{\varepsilon _{0}}\left( x_{0}\right) $, $%
f_{0}^{-1} $ is a lower (or upper) semi-continuous mapping as $%
f_{0}^{-1}:B_{r}^{Y}\left( f_{0}\left( x_{0}\right) \right) \longrightarrow
V\left( x_{0}\right) $ and for $y_{0}=f_{0}\left( x_{0}\right) $ $\in
f\left( x_{0}\right) $ satisfies the following inequations 
\begin{equation*}
\left\{ \left\langle y-y_{0},x-x_{0}\right\rangle \left\vert ~\exists
y\right. \in f\left( x\right) \right\} \geq k\left\langle f_{0}\left(
x\right) -f_{0}\left( x_{0}\right) ,x-x_{0}\right\rangle \geq \nu \left(
\left\Vert x-x_{0}\right\Vert _{X}\right)
\end{equation*}%
\begin{equation*}
\&\ \sup \left\{ \left\Vert y-y_{0}\right\Vert _{X^{\ast }}\left\vert
~y\right. \in f\left( x\right) \right\} \leq \mu \left( \left\Vert
x-x_{0}\right\Vert _{X}\right) ,\ \forall x\in U_{\varepsilon _{0}}\left(
x_{0}\right) ,\ k>0-const.
\end{equation*}%
\begin{equation*}
\nu \in C^{0}\left( R_{+}^{1};R^{1}\right) ,\ \mu \in C^{0}\left(
R_{+}^{1};R_{+}^{1}\right) ,~\exists \varepsilon _{1}:\varepsilon _{0}\geq
\varepsilon _{1}>0,~\nu \left( \varepsilon _{1}\right) \geq \delta >0;
\end{equation*}

(3) There is a nondecreasing continuous function $k_{x_{0}}\left( \tau
\right) \geq 0$ such that 
\begin{equation*}
\left\{ \left\Vert y_{1}-y_{2}-f_{0}\left( x_{1}\right) +f_{0}\left(
x_{2}\right) \right\Vert _{X^{\ast }}~\left\vert ~\forall y_{1}\in f\left(
x_{1}\right) ,\right. ~\exists y_{2}\in f\left( x_{2}\right) -f\left(
x_{1}\right) \right\} \leq
\end{equation*}%
\begin{equation}
k_{x_{0}}\left( \left\Vert f_{0}\left( x_{1}\right) -f_{0}\left(
x_{2}\right) \right\Vert _{X^{\ast }}\right)  \tag{3.4}
\end{equation}%
holds for any pair $x_{1},x_{2}\in V\left( x_{0}\right) $ and $%
k_{x_{0}}^{m_{0}}\left( \tau \right) \leq \sigma \tau $ for some $m_{0}\geq
1 $, $0<\sigma <1$.

\begin{corollary}
Let conditions (1) - (3) are fulfilled. Then there is a number $\delta
=\delta \left( k_{x_{0}}\left( \varepsilon _{0}\right) \right) $ such that a
ball $B_{r_{1}}^{X^{\ast }}\left( f_{0}\left( x_{0}\right) \right) $ belongs
the image $f\left( V\left( x_{0}\right) \right) $ for some $r_{1}\in \left(
0,r\left( 1-\delta \right) \right) $.
\end{corollary}

For the proof it is enough to note that in the conditions of Corollary there
is a continuous selection, which fulfils the conditions of Theorem 4.

Now we will consider the solvability problem globally when the mapping is
single-valued.

(4) Let $f:D\left( f\right) \subseteq X\longrightarrow Y$ be a continuous
mapping, i.e. $f\in C^{0}$;

(5) Let $G\subset X$ be closed convex body of $D\left( f\right) $, $G_{1}$
be subset of $G$, $\overline{G_{1}}\equiv clG_{1}\supseteq G$, there exists
a family of mappings $f_{\xi }:$ $D\left( f_{\xi }\right) \subseteq
X\longrightarrow Y$, $\xi \in G_{1}$ such that $f_{\xi }\left( V_{\xi
}\right) $ is a neighbourhood of $Y$ for some convex body $V_{\xi }\subseteq
D\left( f_{\xi }\right) \cap G\neq \varnothing $, $\left\{ \dbigcup V_{\xi
}\left\vert \ \xi \in G_{1}\right. \right\} \equiv G\subseteq D\left(
f\right) $\ and also\ pair of mappings $f$ and $f_{\xi }$, $\xi \in G_{1}$
satisfies all conditions of Theorem 4 on each subset $V_{\xi }$;

(6) There are a bounded continuous mapping $g:D\left( g\right) \subseteq
X\longrightarrow Y^{\ast }$, $\left\{ \dbigcup V_{\xi }\left\vert \ \xi \in
G_{1}\right. \right\} $ $\subseteq D\left( g\right) $ and an element $%
x_{0}\in int~G$ such that $g\left( G\right) \equiv B_{1}^{Y^{\ast }}\left(
0\right) $, $g\left( x_{0}\right) =0$ and 
\begin{equation*}
\left\langle f\left( x\right) -f\left( x_{0}\right) ,g\left( x\right)
\right\rangle \geq \nu \left( \left\Vert x-x_{0}\right\Vert _{X}\right) ,
\end{equation*}%
holds for any $x\in G$, and pair of mappings $f$, $g$ and the function $\nu
\nearrow $ satisfies all conditions of Theorem 1 also.

\begin{theorem}
Let $X$ and $Y$ are reflexive Banach spaces, and let the conditions (4) -
(6) are fulfilled on $G\subseteq D\left( f\right) \subseteq X$. Then the
image $f\left( G\right) $ is bodily subset of $Y$ and belongs the following
subset 
\begin{equation*}
M\equiv \left\{ y\in Y\left\vert ~\left\langle y,g\left( x\right)
\right\rangle \leq \left\langle f\left( x\right) ,g\left( x\right)
\right\rangle \right. \forall x\in G_{0}\subseteq g^{-1}\left(
S_{1}^{Y^{\ast }}\left( 0\right) \right) \right\}
\end{equation*}%
defined by $G_{0}$ where $G_{0}\subseteq G$ and $g\left( G_{0}\right) \equiv 
$ $S_{1}^{Y^{\ast }}\left( 0\right) $.
\end{theorem}

\begin{proof}
Under conditions of Theorem 5 follows, that $f\left( G\right) $ contains a
subset which is at least everywhere dense in some bodily subset of $M$ by
virtue of Theorem 1.

Now with using of Theorem 4 we obtain, that $f\left( V_{\xi }\right) $ is
bodily subset of $Y$ and consequently $f\left( G\right) $ is the bodily set
of $Y$ as far as $f\left( G\right) \subseteq \left\{ \dbigcup f\left( V_{\xi
}\right) \left\vert \ \xi \in G_{1}\right. \right\} $. Thus for the
completion of this proof is remained to use the proof of Theorem 1.
\end{proof}

The following statement immediately follows from Theorem 5.

\begin{corollary}
Let all conditions of Theorem 5 are fulfilled on $D\left( f\right) \equiv X$%
, and also the following relation holds 
\begin{equation*}
\nu \left( \left\Vert x\right\Vert _{X}\right) \nearrow \infty ,\quad \text{%
at \ }\left\Vert x\right\Vert _{X}\nearrow \infty .
\end{equation*}

Then $f\left( X\right) =Y$, i.e. $f$ is surjection.
\end{corollary}

\textbf{Now we will consider an other type of the comparison.}

a) Let $f:D\left( f\right) \subseteq X\longrightarrow X^{\ast }$ be a
bounded continuous mapping (i.e. $f\in C^{0}$) and $B_{r}^{X}\left( 0\right)
\subseteq D\left( f\right) $ is a closed ball, where $r>0$;

b) For each $x_{0}\in G\subseteq B_{r}^{X}\left( 0\right) $, $clG\equiv
B_{r}^{X}\left( 0\right) $ there are a neighbourhood $U_{\varepsilon }\left(
x_{0}\right) $ and a coercive demicontinuous monotone operators $%
f_{ix_{0}}:X\longrightarrow X^{\ast }$, ($i=1,2$) ([4-6, 21, 30]) such that 
\begin{equation*}
\left\langle f_{1x_{0}}\left( x\right) ,x\right\rangle \geq \left\langle
f\left( x\right) ,x\right\rangle \geq \left\langle f_{2x_{0}}\left( x\right)
,x\right\rangle \geq \nu \left( \left\Vert x\right\Vert _{X}\right)
\left\Vert x\right\Vert _{X}
\end{equation*}%
for any $x\in U_{\varepsilon }\left( x_{0}\right) $, where $\varepsilon \geq
\varepsilon _{0}>0$ and $\nu :R_{+}\longrightarrow R^{1}$ is a continuous
function as above, moreover the inequalities 
\begin{equation*}
\left\vert \left\langle f_{1x_{0}}\left( x_{0}\right) ,y\right\rangle
\right\vert \geq \left\vert \left\langle f\left( x_{0}\right)
,y\right\rangle \right\vert \geq \left\vert \left\langle f_{2x_{0}}\left(
x_{0}\right) ,y\right\rangle \right\vert
\end{equation*}%
hold for any $y\in S_{1}^{X}\left( 0\right) $;

c) There are a continuous function $k_{x_{0}}\left( x_{1},x_{2}\right) >0$
and a numbers $\lambda _{x_{0}},\rho _{x_{0}}\geq 0$ such that 
\begin{equation*}
\left\vert \left\langle f\left( x_{1}\right) -f\left( x_{2}\right)
-f_{0x_{0}}\left( x_{1}\right) +f_{0x_{0}}\left( x_{2}\right)
,y\right\rangle \right\vert \leq k_{x_{0}}\left( x_{1},x_{2}\right)
~\left\Vert f_{0x_{0}}\left( x_{1}\right) -f_{0x_{0}}\left( x_{2}\right)
\right\Vert _{X^{\ast }}
\end{equation*}%
for any $x_{1},x_{2}\in U_{\varepsilon }\left( x_{0}\right) $, $\forall y\in
S_{1}^{X}\left( 0\right) $ where $f_{0x_{0}}\left( x\right) \equiv \lambda
_{x_{0}}f_{1x_{0}}\left( x\right) +\rho _{x_{0}}f_{2x_{0}}\left( x\right) $,
\ $\forall x\in U_{\varepsilon }\left( x_{0}\right) $ and $\lambda
_{x_{0}}+\rho _{x_{0}}=1$.

\begin{theorem}
Let $X$ be a reflexive Banach space and the mapping $f$ satisfies conditions
a) - c). Then if $k_{x_{0}}\left( x_{1},x_{2}\right) <1$ then the statement
of Theorem 5 is true for the mapping $f$ on $B_{r}^{X}\left( 0\right) $.
\end{theorem}

For the proof it is enough to show that all conditions of Theorem 5 are
fulfilled under conditions of this theorem.

\begin{remark}
If in this theorem we have $D\left( f\right) \equiv X$ and all conditions of
this theorem is fulfilled on $X$ then the statement of Theorem 6 is true on $%
X$, in addition if the condition of the previous Corollary is fulfilled then 
$f\left( X\right) \equiv X^{\ast }$.
\end{remark}

\section{\protect\bigskip On Applications of the General Results to the
Nonlinear Boundary Problems}

\subsection{Dirichlet problem for quasilinear elliptic equations\textbf{\ }}

Let $\Omega \subset R^{d}$, $d\geq 1$ be a bounded domain with sufficiently
smooth boundary $\partial \Omega $. We will consider the following problem 
\begin{eqnarray}
f\left( u\right) &\equiv &-\underset{i=1}{\overset{d}{\sum }}~D_{i}\left(
a_{i}\left( x,u,Du\right) D_{i}u\right) =h\left( x\right) ,\quad x\in \Omega
\TCItag{4.1} \\
u\left\vert ~_{\partial \Omega }\right. &=&0,\qquad D_{i}\equiv \frac{%
\partial }{\partial x_{i}},\ D\equiv \left( D_{1},...,D_{d}\right) .  \notag
\end{eqnarray}%
here $a_{i}\left( x,\xi ,\eta \right) $ ($i=\overline{1,d}$), $h\left(
x\right) $ be some functions.

We consider the following conditions.

7) Let $a_{i}\left( x,\xi ,\eta \right) $ ($i=\overline{1,d}$) be a
Caratheodory function for $\left( x,\xi ,\eta \right) \in \Omega \times
R^{d+1}$, and there exists a continuous function $\varphi
:R^{1}\longrightarrow R_{+}^{1}$ such that $\Phi :L_{p}\left( \Omega \right)
\longrightarrow L_{q}\left( \Omega \right) $ is a bounded continuous
mapping, $p,q\geq 1$, where $\Phi \left( \xi \right) \equiv \underset{0}{%
\overset{\xi }{\dint }}\varphi \left( \tau \right) d\tau $;

8) For a.e. $u_{0}\in \overset{0}{W}~_{p}^{1}\left( \Omega \right) $ there
exist a functions $b_{i}^{u_{0}}\left( x\right) $, $c_{i}^{u_{0}}\left(
x\right) \geq 0$, a number $\varepsilon \left( u_{0}\right) \geq \varepsilon
_{0}>0$ and a neighbourhood $B_{\varepsilon \left( u_{0}\right)
}^{W_{p}^{1}}\left( u_{0}\left( x\right) \right) $ such that 
\begin{equation*}
~b_{i}^{u_{0}}\left( x\right) \varphi \left( \eta _{i}\right) \geq
a_{i}\left( x,\xi ,\eta \right) \geq c_{i}^{u_{0}}\left( x\right) \varphi
\left( \eta _{i}\right)
\end{equation*}%
for any $\left( \xi ,\eta \right) \in R^{1}\times R^{d}$ and a.e. $x\in
\Omega $;

From here follows, that the following expressions generates a monotone
operators on a respective spaces 
\begin{equation*}
f_{1}^{u_{0}}\left( u\right) \equiv -\underset{i=1}{\overset{d}{\sum }}~D_{i}%
\left[ b_{i}^{u_{0}}\left( x\right) \varphi \left( D_{i}u\right) D_{i}u%
\right] ,~\&
\end{equation*}%
\begin{equation*}
f_{2}^{u_{0}}\left( u\right) \equiv -\underset{i=1}{\overset{d}{\sum }}~D_{i}%
\left[ c_{i}^{u_{0}}\left( x\right) \varphi \left( D_{i}u\right) D_{i}u%
\right] .
\end{equation*}

9) there exist a continuous function $k\left( u_{0},\xi ,\eta \right) >0$
and numbers $\lambda \left( u_{0}\right) $, $\rho \left( u_{0}\right) \geq 0$
such that $\lambda \left( u_{0}\right) +\rho \left( u_{0}\right) =1$ and 
\begin{equation*}
a_{i}\left( x,\xi ,\eta \right) \eta _{i}-a_{i}\left( x,\widehat{\xi },%
\widehat{\eta }\right) \widehat{\eta }_{i}=k\left( u_{0},\xi ,\eta ,\widehat{%
\xi },\widehat{\eta }\right) \left[ \varphi \left( \eta _{i}\right) \eta
_{i}-\varphi \left( \widehat{\eta }_{i}\right) \widehat{\eta }_{i}\right]
\end{equation*}%
and 
\begin{equation*}
\left\vert k\left( u_{0},\xi ,\eta ,\widehat{\xi },\widehat{\eta }\right)
-\lambda \left( u_{0}\right) b_{i}^{u_{0}}\left( x\right) -\rho \left(
u_{0}\right) c_{i}^{u_{0}}\left( x\right) \right\vert <1
\end{equation*}%
for any $\left( \xi ,\eta \right) ,\left( \widehat{\xi },\widehat{\eta }%
\right) \in R^{1}\times R^{d}$.

\begin{theorem}
Let conditions 7) - 9) are fulfilled then the problem (4.1) is solvable in
the space $\overset{0}{W}~_{p}^{1}\left( \Omega \right) $ for any $h\in
W_{q}^{-1}\left( \Omega \right) $, $q=\frac{p}{p-1}$.
\end{theorem}

\begin{proof}
For the proof it is enough to show that all conditions of Theorem 6 are
fulfilled. Namely it is enough to show that condition c) is true for the
mapping $f:\overset{0}{W}~_{p}^{1}\left( \Omega \right) \longrightarrow
W_{q}^{-1}\left( \Omega \right) $ determined by problem (4.1) and the lather
defined system of mappings, if exactly for the mappings 
\begin{equation*}
f,\left\{ f_{1}^{u_{0}}\left\vert ~u_{0}\in \right. M\right\} ,\left\{
f_{2}^{u_{0}}\left\vert ~u_{0}\in \right. M\right\} :\overset{0}{W}%
~_{p}^{1}\left( \Omega \right) \longrightarrow W_{q}^{-1}\left( \Omega
\right)
\end{equation*}%
where $M\subseteq \overset{0}{W}~_{p}^{1}\left( \Omega \right) $ and $%
\overline{M}^{W_{p}^{1}}\equiv \overset{0}{W}~_{p}^{1}\left( \Omega \right) $%
, as the fulfilment of the other conditions are obviously by virtue of the
selection of the systems $\left\{ f_{j}^{u_{0}}\left\vert ~u_{0}\in \right.
M,\ j=1,2\right\} $.

Now we will define the above systems. Let $u_{0}\in \overset{0}{W}%
~_{p}^{1}\left( \Omega \right) $ be such element as in the condition 8).
Thus if we choose the mapping $f_{0}^{u_{0}}$\ on the ball $B_{\varepsilon
\left( u_{0}\right) }^{W_{p}^{1}}\left( u_{0}\left( x\right) \right) $, for
any $u_{0}\in M$ of the following form 
\begin{equation*}
f_{0}^{u_{0}}\left( u\right) \equiv -\underset{i=1}{\overset{d}{\sum }}~D_{i}%
\left[ a_{i}^{u_{0}}\left( x,u,Du\right) D_{i}u\right] ,
\end{equation*}%
where 
\begin{equation*}
a_{i}^{u_{0}}\left( x,\xi ,\eta \right) \eta _{i}=\lambda \left(
u_{0}\right) b_{i}^{u_{0}}\left( x\right) \varphi \left( \eta _{i}\right)
\eta _{i}+\rho \left( u_{0}\right) c_{i}^{u_{0}}\left( x\right) \varphi
\left( \eta _{i}\right) \eta _{i}=
\end{equation*}%
\begin{equation*}
\left[ \lambda \left( u_{0}\right) \widetilde{b}_{i}^{u_{0}}\left( x\right)
+\rho \left( u_{0}\right) \widetilde{c}_{i}^{u_{0}}\left( x\right) \right]
\varphi \left( \eta _{i}\right) \eta _{i}
\end{equation*}%
then we have $f_{0}^{u_{0}}:$ $\overset{0}{W}~_{p}^{1}\left( \Omega \right)
\longrightarrow W_{q}^{-1}\left( \Omega \right) $ is a continuous mapping
and the condition c) is fulfilled as 
\begin{equation*}
\left[ a_{i}\left( x,\xi ,\eta \right) \eta _{i}-a_{i}\left( x,\widehat{\xi }%
,\widehat{\eta }\right) \widehat{\eta }_{i}\right] -\left[
a_{i}^{u_{0}}\left( x,\xi ,\eta \right) \eta _{i}-a_{i}^{u_{0}}\left( x,%
\widehat{\xi },\widehat{\eta }\right) \widehat{\eta }_{i}\right] =
\end{equation*}%
\begin{equation*}
\left[ k\left( u_{0},\xi ,\eta ,\widehat{\xi },\widehat{\eta }\right)
-\lambda \left( u_{0}\right) b_{i}^{u_{0}}\left( x\right) -\rho \left(
u_{0}\right) c_{i}^{u_{0}}\left( x\right) \right] \left[ \varphi \left( \eta
_{i}\right) \eta _{i}-\varphi \left( \widehat{\eta }_{i}\right) \widehat{%
\eta }_{i}\right] .
\end{equation*}%
Really from here we obtain the fulfillment of the condition c) with using
the conditions 8), 9). Consequently all conditions of Theorem 6 hold for the
mapping $f$ with $\left\{ f_{0}^{u_{0}}\left\vert ~u_{0}\in \right.
M\right\} $ what proved this Theorem.
\end{proof}

It should be noted that with using this approach can be considered more
general nonlinear boundary problem than the problem (4.1).

\subsection{On some modification of the Navier-Stokes equations}

Now we will consider the following problem on the $Q\equiv \Omega \times
\left( 0,T\right) $ 
\begin{equation}
\frac{\partial u}{\partial t}-\nu \Delta u-div~\left[ 2\varphi \left(
s\left( u\right) \right) Bu\right] +\overset{3}{\underset{i=1}{\sum }}%
~u_{i}D_{i}u+grad~p=h\left( t,x\right) ,  \tag{4.2}
\end{equation}%
\begin{equation}
div~u=0,\quad u\left\vert ~_{\partial \Omega \times \left[ 0,T\right]
}\right. =0,\quad u\left( 0,x\right) =0,\quad x\in \Omega \subset R^{3}, 
\tag{4.3}
\end{equation}%
where the domain $\Omega \subset R^{3}$ is such as above, $u\equiv \left(
u_{1},u_{2},u_{3}\right) $, $\nu >0$ is the known physics constant, $\Delta $
is a Laplacian, $h\left( t,x\right) $ is some function and 
\begin{equation*}
B\left( u\right) \equiv \left( b_{ij}\left( u\right) \right)
_{i,j=1,2,3},b_{ij}\left( u\right) \equiv \frac{1}{2}\left(
D_{i}u_{j}+D_{j}u_{i}\right) ,~s\left( u\right) \equiv \left[ \overset{3}{%
\underset{i,j=1}{\sum }}b_{ij}\left( u\right) \right] ^{\frac{1}{2}}.
\end{equation*}

As usually we introduce the following known space ([24]) 
\begin{equation*}
V\left( \Omega \right) \equiv \left\{ u\left( x\right) \left\vert ~u\in
\left( \overset{0}{W}~_{2}^{1}\left( \Omega \right) \right) ^{3},\quad
div~u=0\right. \right\} ,~u\equiv \left( u_{1},u_{2},u_{3}\right) ,
\end{equation*}%
\begin{equation*}
L_{2}\left( 0,T;V\left( \Omega \right) \right) \equiv \left\{ u\left(
t,x\right) \left\vert ~u_{i}\in L_{2}\left( 0,T;\overset{0}{W}%
~_{2}^{1}\left( \Omega \right) \right) \right. ,~div~u=0\right\} .
\end{equation*}

It should be noted that the modification of the Navier-Stokes equations of
such type earlier was suggested and investigated under some conditions on
the third term by Ladyzhenskaya [23], after the problem of such type were
studied by other authors [23, 27, 12, 31] etc. In these works the third term
are selected such that, the considered problem were uniquely solvable,
furthermore essentially are studied certain stationary case of such problem.
Moreover Navier-Stokes equations and their various modification were studied
and now many authors conduct certain investigation of the problems of such
type, which are connected with some problem for the different modification
of the Navier-Stokes equations (see, for example in the [14, 9, 29] and
their references). Here we investigate this problem in the stationary and
nonstationary cases, and else under the different conditions on the third
term.

So let the following conditions are fulfilled.

10) $\varphi \in C^{0}$ and there is a number $\delta >0$ such that if $u\in
L_{2}\left( 0,T;V\left( \Omega \right) \right) $ then $\left\Vert \varphi
\left( s\left( u\right) \right) \right\Vert _{L^{\infty }}\leq \frac{\nu }{3}%
-\delta $;

11) there is a number $\mu \geq 0$ such that 
\begin{equation*}
\left\vert \varphi \left( t\right) -\varphi \left( \tau \right) \right\vert
\left\vert t\right\vert \leq \left\vert \mu +\varphi \left( t\right)
\right\vert ~\left\vert t-\tau \right\vert ,\quad 0\leq \mu \leq \frac{\nu }{%
3}-\delta
\end{equation*}%
for any $t,\tau \in R^{1}$;

For instance, $\varphi \left( \tau \right) =c_{0}\tau ^{\rho }\left( \tau
^{\sigma }+c_{1}\right) ^{-1}$, for $\tau \geq 0$ where $\rho ,\sigma \geq 1$%
, $\sigma \geq \rho $ and $c_{i}\geq 0$, $i=0,1$.

\subsubsection{The stationary case.}

In this case the considered problem is equivalent to the following
functional equation 
\begin{equation}
\left\langle f\left( u\right) ,v\right\rangle \equiv \left\langle -\nu
\Delta u-div~\left[ 2\varphi \left( s\left( u\right) \right) Bu\right] +%
\overset{3}{\underset{i=1}{\sum }}~u_{i}D_{i}u,v\right\rangle =\left\langle
h,v\right\rangle  \tag{4.2`}
\end{equation}%
for any $v\in V\left( \Omega \right) $.

For investigation we will use Theorem 2 (really - Lemma 1 and Corollary 1)
in the case when the considered mapping is single-valued. Therefore it is
needed to show fulfillment of all conditions of Theorem 2 for the operator
generated by this problem, here to this end the perturbation $f_{1}$ is
determined in the form 
\begin{equation*}
f_{1}\left( u\right) \equiv \left( f_{1}^{1}\left( u\right) ,f_{1}^{2}\left(
u\right) ,f_{1}^{3}\left( u\right) \right) \equiv \left( \overset{3}{%
\underset{i=1}{\sum }}~u_{i}D_{i}u_{1},\overset{3}{\underset{i=1}{\sum }}%
~u_{i}D_{i}u_{2},\overset{3}{\underset{i=1}{\sum }}~u_{i}D_{i}u_{3}\right) .
\end{equation*}

It is easly to see that 
\begin{equation*}
\left\langle f\left( u\right) ,u\right\rangle \equiv \left\langle -\nu
\Delta u-div~\left[ 2\varphi \left( s\left( u\right) \right) Bu\right] +%
\overset{3}{\underset{i=1}{\sum }}~u_{i}D_{i}u,u\right\rangle =
\end{equation*}%
\begin{equation*}
\nu \left\Vert u\right\Vert _{V\left( \Omega \right) }^{2}+\left\langle
\varphi \left( s\left( u\right) \right) Bu,Bu\right\rangle
\end{equation*}%
holds for any $u\in V\left( \Omega \right) $. And from here with use the
condition 10 
\begin{equation*}
\left\langle f\left( u\right) ,u\right\rangle \geq \frac{2\nu }{3}\left\Vert
u\right\Vert _{V\left( \Omega \right) }^{2},\quad \forall u\in V\left(
\Omega \right) .
\end{equation*}%
Consequently we need to prove fulfillment of the second inequation of
Theorem 2. First we will consider the second term of the left part (4.2`) 
\begin{equation*}
\left\langle -div~\left[ 2\varphi \left( s\left( u\right) \right) Bu\right]
+div~\left[ 2\varphi \left( s\left( v\right) \right) Bv\right]
,u-v\right\rangle =
\end{equation*}%
\begin{equation*}
\left\vert \left\langle \varphi \left( s\left( u\right) \right) Bu-\varphi
\left( s\left( v\right) \right) Bv,B\left( u-v\right) \right\rangle
\right\vert \leq \left\vert \left\langle \left[ \varphi \left( s\left(
u\right) \right) -\varphi \left( s\left( v\right) \right) \right] Bu,B\left(
u-v\right) \right\rangle \right\vert +
\end{equation*}%
\begin{equation*}
\left\vert \left\langle \varphi \left( s\left( v\right) \right) B\left(
u-v\right) ,B\left( u-v\right) \right\rangle \right\vert \text{ and by
condition 11 }\leq
\end{equation*}%
\begin{equation*}
\leq \left\langle \left\vert \mu +\varphi \left( s\left( u\right) \right)
\right\vert B\left( u-v\right) ,B\left( u-v\right) \right\rangle +\left\Vert
\varphi \left( s\left( v\right) \right) \right\Vert _{L^{\infty }}\left\Vert
B\left( u-v\right) \right\Vert _{2}^{2}\leq
\end{equation*}%
\begin{equation}
\left[ \mu +2\left\Vert \varphi \left( s\left( u\right) \right) \right\Vert
_{L^{\infty }}\right] \left\Vert B\left( u-v\right) \right\Vert _{2}^{2}\leq
\left( \nu -3\delta \right) \left\Vert Bw\right\Vert _{\left( L_{2}\right)
^{3}}^{2},~\forall u\in V\left( \Omega \right) .  \tag{*}
\end{equation}

Now we will show a lower estimate for the expression $f(u)-f(v)$ with use of
the inequation (*) 
\begin{equation*}
\left\Vert f(u)-f(v)\right\Vert _{V^{\ast }}\equiv \left\Vert -\nu \Delta
\left( u-v\right) -div\left[ 2\varphi \left( s\left( u\right) \right)
Bu-2\varphi \left( s\left( v\right) \right) Bv\right] \right. +
\end{equation*}%
\begin{equation*}
\left. \overset{3}{\underset{i=1}{\sum }}\left[ u_{i}D_{i}u-v_{i}D_{i}v%
\right] \right\Vert _{V^{\ast }}\geq \nu \left\Vert u-v\right\Vert
_{V}-\left\Vert div\left[ 2\varphi \left( s\left( u\right) \right)
Bu-2\varphi \left( s\left( v\right) \right) Bv\right] \right\Vert _{V^{\ast
}}-
\end{equation*}%
\begin{equation*}
\left\Vert \overset{3}{\underset{i=1}{\sum }}~\left[ u_{i}u-v_{i}v\right]
\right\Vert _{\left( L_{2}\right) ^{3}}\geq \nu \left\Vert w\right\Vert
_{V}-\left( \nu -3\delta \right) \left\Vert Bw\right\Vert _{\left(
L_{2}\right) ^{3}}-
\end{equation*}%
\begin{equation*}
\frac{1}{2}\left( \left\Vert u\right\Vert _{\left( L_{4}\right)
^{3}}+\left\Vert v\right\Vert _{\left( L_{4}\right) ^{3}}\right) \left\Vert
w\right\Vert _{\left( L_{4}\right) ^{3}}=
\end{equation*}%
\begin{equation}
3\delta \left\Vert w\right\Vert _{V}-\frac{1}{2}\left( \left\Vert
u\right\Vert _{\left( L_{4}\right) ^{3}}+\left\Vert v\right\Vert _{\left(
L_{4}\right) ^{3}}\right) \left\Vert w\right\Vert _{\left( L_{4}\right)
^{3}}.  \tag{4.4}
\end{equation}

From here follows, that the condition (\textit{iii}) of Theorem 2 is
satisfied for the considered problem. Cosequently we obtained, that all
conditions of Theorem 2 are fulfilled in case of Corollary 1. So we will
show that the image $f\left( B_{r}\left( 0\right) \right) $ is a closed
subset of $V^{\ast }\left( \Omega \right) $ with use of this inequality. It
should be noted that the imbedding $\overset{0}{W}~_{2}^{1}\left( \Omega
\right) \subset L_{4}\left( \Omega \right) $ is compact, and from (4.4)
follows 
\begin{equation}
\left\Vert f(u)-f(v)\right\Vert _{V^{\ast }}+\frac{1}{2}\left( \left\Vert
u\right\Vert _{\left( L_{4}\right) ^{3}}+\left\Vert v\right\Vert _{\left(
L_{4}\right) ^{3}}\right) \left\Vert w\right\Vert _{\left( L_{4}\right)
^{3}}\geq 3\delta \left\Vert w\right\Vert _{V}.  \tag{4.5}
\end{equation}

Let $\left\{ h_{m}\right\} _{m=1}^{\infty }\subset f\left( B_{r}\left(
0\right) \right) $ and $h_{m}\Longrightarrow h$ in $V^{\ast }\left( \Omega
\right) $, then $f^{-1}\left( \left\{ h_{m}\right\} _{m=1}^{\infty }\right)
\subseteq $ $B_{r}\left( 0\right) $ and is bounded subset of $V\left( \Omega
\right) $. Consequently there exists sequence $\left\{ u_{m_{k}}\right\}
_{k=1}^{\infty }\subset $ $B_{r}\left( 0\right) $ such that $f\left(
u_{m_{k}}\right) =h_{m_{k}}$ and $u_{m_{k}}\rightharpoonup u$ in $V\left(
\Omega \right) $ and therefore $u_{m_{k}}\Longrightarrow u$ in $\left(
L_{4}\left( \Omega \right) \right) ^{3}$. Hence we obtain, that the image $%
f\left( B_{r}\left( 0\right) \right) $ is a closed subset of $V^{\ast
}\left( \Omega \right) $ by virtue of (4.5).

Thus the following solvability theorem is proved.

\begin{theorem}
Let conditions 10 and 11 are fulfilled then the stationary Navier-Stokes
problem of the (4.2)-(4.3) type is solvable in the $V\left( \Omega \right)
\times L_{2}\left( \Omega \right) $ for any $h\in V^{\ast }\left( \Omega
\right) $.
\end{theorem}

\begin{remark}
It should be noted that a theorem of such type is true under the following
conditions.

12) let $\varphi \in C^{0}$ , $\varphi \left( \tau \right) \geq 0$ and $%
\left\Vert \varphi \left( s\left( u\right) \right) \right\Vert _{L^{\infty
}}\leq \mu $ for any $u\in V\left( \Omega \right) $ and $\tau \in R^{1}$, $%
\mu >0$ is a constant.

13) $\left\langle \varphi \left( s\left( u\right) \right) Bu-\varphi \left(
s\left( v\right) \right) Bv,B\left( u-v\right) \right\rangle \geq 0$ for any 
$u,v\in V\left( \Omega \right) $.
\end{remark}

\subsubsection{The problem (4.2)-(4.3)}

We denote by $F:W\left( Q\right) \longrightarrow L_{2}\left( 0,T;V^{\ast
}\left( \Omega \right) \right) $ the operator generated by problem
(4.2)-(4.3) 
\begin{equation*}
F\left( u\right) \equiv \frac{\partial u}{\partial t}+f\left( u\right)
,\quad \forall u\in W\left( Q\right) \cap \left\{ u\left\vert ~u\left(
0\right) \right. =0\right\}
\end{equation*}%
here 
\begin{equation*}
W\left( Q\right) \equiv \left( W_{2}^{1}\left( 0,T;L_{2}\left( \Omega
\right) \right) \cap L_{2}\left( 0,T;\overset{0}{W}~_{2}^{1}\left( \Omega
\right) \right) \right) ^{3}\cap \left\{ u\left\vert \ div\right.
u=0\right\} .
\end{equation*}%
It is clear that $W\left( Q\right) \subset \left( L_{2}\left( 0,T;\overset{0}%
{W}~_{2}^{1}\left( \Omega \right) \right) \right) ^{3}\cap \left\{
u\left\vert \ div\right. u=0\right\} $ and this inclusion is everywhere
dense. Consequently we will use $W\left( Q\right) $ as an everywhere dense
space in $L_{2}\left( 0,T;V^{\ast }\left( \Omega \right) \right) $.

Further we have the following inequation 
\begin{equation*}
\left\langle \frac{\partial u}{\partial t}-\nu \Delta u-div~\left[ 2\varphi
\left( s\left( u\right) \right) Bu\right] +\overset{3}{\underset{i=1}{\sum }}%
~u_{i}D_{i}u,u\right\rangle =
\end{equation*}%
\begin{equation*}
\frac{1}{2}\frac{\partial }{\partial t}\left\Vert u\right\Vert _{2}^{2}+\nu
\left\Vert \nabla u\right\Vert _{2}^{2}+\left\langle \varphi \left( s\left(
u\right) \right) Bu,Bu\right\rangle \geq \frac{1}{2}\frac{\partial }{%
\partial t}\left\Vert u\right\Vert _{2}^{2}+
\end{equation*}%
\begin{equation*}
\nu \left\Vert \nabla u\right\Vert _{2}^{2}-\left( \frac{\nu }{3}-\delta
\right) \left\Vert Bu\right\Vert _{2}^{2}=\frac{1}{2}\frac{\partial }{%
\partial t}\left\Vert u\right\Vert _{2}^{2}+\left( \frac{2\nu }{3}+\delta
\right) \left\Vert \nabla u\right\Vert _{2}^{2}
\end{equation*}%
holds for any $u\in W\left( Q\right) $. From here follows 
\begin{equation*}
\underset{0}{\overset{t}{\int }}\left\langle \frac{\partial u}{\partial t}%
-\nu \Delta u-div~\left[ 2\varphi \left( s\left( u\right) \right) Bu\right] +%
\overset{3}{\underset{i=1}{\sum }}~u_{i}D_{i}u,u\right\rangle d\tau \geq
\end{equation*}%
\begin{equation*}
\frac{1}{2}\left\Vert u\right\Vert _{2}^{2}\left( t\right) +\left( \frac{%
2\nu }{3}+\delta \right) \underset{0}{\overset{t}{\int }}\left\Vert \nabla
u\right\Vert _{2}^{2}\left( \tau \right) d\tau .
\end{equation*}

Now we will estimate the expression $\left\Vert f(u)-f(v)\right\Vert
_{L_{2}\left( V^{\ast }\right) }$ for any $u,v\in W\left( Q\right) $ with
use of the inequation (*)%
\begin{equation*}
\underset{0}{\overset{T}{\int }}\left\vert \left\langle
f(u)-f(v),w\right\rangle \right\vert dt=\underset{0}{\overset{T}{\int }}%
\left\vert \left\langle \frac{\partial \left( u-v\right) }{\partial t}-\nu
\Delta \left( u-v\right) ,w\right\rangle -\right.
\end{equation*}%
\begin{equation*}
\left. \left\langle div\left[ 2\varphi \left( s\left( u\right) \right)
Bu-2\varphi \left( s\left( v\right) \right) Bv\right] +\overset{3}{\underset{%
i=1}{\sum }}~\left[ u_{i}D_{i}u-v_{i}D_{i}v\right] ,w\right\rangle
\right\vert dt\geq
\end{equation*}%
\begin{equation*}
\left\vert \underset{0}{\overset{T}{\int }}\left\langle \frac{\partial w}{%
\partial t}-\nu \Delta w,w\right\rangle \left( t\right) dt\right\vert -
\end{equation*}%
\begin{equation*}
\underset{0}{\overset{T}{\int }}\left\vert \left\langle div\left[ 2\varphi
\left( s\left( u\right) \right) Bu-2\varphi \left( s\left( v\right) \right)
Bv\right] +\overset{3}{\underset{i=1}{\sum }}~\left[ u_{i}D_{i}u-v_{i}D_{i}v%
\right] ,w\right\rangle \right\vert dt\geq
\end{equation*}%
by virtue of the condition \textit{11} of this example, which can be shown
analogously as (*) 
\begin{equation*}
\frac{1}{2}\underset{0}{\overset{T}{\int }}\frac{\partial }{\partial t}%
\left\Vert w\right\Vert _{2}^{2}\left( t\right) dt+\nu \underset{0}{\overset{%
T}{\int }}\left\Vert w\right\Vert _{V}^{2}~\left( t\right) dt-
\end{equation*}%
\begin{equation*}
\left( \nu -3\delta \right) \left\Vert w\right\Vert _{L_{2}\left( V\right)
}^{2}~-\frac{1}{2}\underset{0}{\overset{T}{\int }}\left( \left\Vert
u\right\Vert _{\left( L_{3}\right) ^{3}}~\left( t\right) +\left\Vert
v\right\Vert _{\left( L_{3}\right) ^{3}}~\left( t\right) \right) \left\Vert
w\right\Vert _{\left( L_{3}\right) ^{3}}^{2}~\left( t\right) dt.
\end{equation*}

It should be noted that from $u\in \left( L^{\infty }\left( 0,T;L_{2}\left(
\Omega \right) \right) \right) ^{3}\cap $ $L_{2}\left( 0,T;V\left( \Omega
\right) \right) $ follows $u_{i}\in L_{4}\left( 0,T;L_{3}\left( \Omega
\right) \right) $, therefore if we continue this estimate, then we have 
\begin{equation*}
\left\Vert f(u)-f(v)\right\Vert _{W_{2}^{-1}\left( V^{\ast }\right)
}\left\Vert w\right\Vert _{L_{2}\left( V\right) }\geq \frac{1}{2}\underset{0}%
{\overset{T}{\int }}\frac{\partial }{\partial t}\left\Vert w\right\Vert
_{2}^{2}\left( t\right) dt+3\delta \underset{0}{\overset{T}{\int }}%
\left\Vert w\right\Vert _{V}^{2}~\left( t\right) dt-
\end{equation*}%
\begin{equation*}
-\frac{1}{2}\underset{0}{\overset{T}{\int }}\left( \left\Vert u\right\Vert
_{\left( L_{3}\right) ^{3}}~\left( t\right) +\left\Vert v\right\Vert
_{\left( L_{3}\right) ^{3}}~\left( t\right) \right) \left\Vert w\right\Vert
_{\left( L_{3}\right) ^{3}}^{2}~\left( t\right) dt.
\end{equation*}

Thus we obtain the following inequation 
\begin{equation*}
\left\Vert f(u)-f(v)\right\Vert _{W_{2}^{-1}\left( V^{\ast }\right) }\geq 
\frac{1}{2}\left\Vert w\right\Vert _{2}\left( T\right) +3\delta \left\Vert
w\right\Vert _{L_{2}\left( V\right) }-
\end{equation*}%
\begin{equation*}
\left( \left\Vert u\right\Vert _{L_{4}\left( \left( L_{3}\right) ^{3}\right)
}+\left\Vert v\right\Vert _{L_{4}\left( \left( L_{3}\right) ^{3}\right)
}\right) \left\Vert w\right\Vert _{L_{p_{0}}\left( \left( L_{3}\right)
^{3}\right) },\ p_{0}=\frac{8}{3}.
\end{equation*}

From here follows, that all conditions of Theorem 2 and Corollary1 are
fulfilled for the considered problem where $F_{1}$ is determined as in above
subsection. Then we can use Theorem 2 in case of Corollary 1 for this
problem. So by using we obtain the correctness of the following statement.

\begin{theorem}
Under the conditions of this subsection the problem (4.2)-(4.3) is solvable
in $W\left( Q\right) \times L_{2}\left( Q\right) $ for any $h\in L_{2}\left(
0,T;V^{\ast }\left( \Omega \right) \right) $.
\end{theorem}

\end{document}